\documentclass[a4paper,11pt,reqno]{amsart}

\usepackage{
amsfonts,
amsmath,
amsopn,
amssymb,
amsthm,
bbm,
bbold,
dsfont,
enumitem,
graphicx,
mathrsfs,
mathtools,
soul,
subfig,
verbatim,
xcolor,
xspace
}

\usepackage{geometry}
\usepackage[hidelinks]{hyperref}
\usepackage[utf8]{inputenc}

\mathtoolsset{showonlyrefs}

\newtheorem{theorem}{Theorem}[section]
\newtheorem{lemma}[theorem]{Lemma}
\newtheorem{proposition}[theorem]{Proposition}
\newtheorem{corollary}[theorem]{Corollary}

\theoremstyle{remark}
\newtheorem{remark}[theorem]{Remark}
\newtheorem*{remark*}{Remark}

\theoremstyle{definition}
\newtheorem{definition}[theorem]{Definition}

\newcommand{\R}{\mathbb{R}}
\newcommand{\Rd}{\mathbb{R}^3}

\newcommand{\C}{\mathbb{C}}

\newcommand{\N}{\mathbb{N}}

\newcommand{\Eps}{\mathcal{E}}
\newcommand{\G}{\mathcal{G}}

\renewcommand{\leq }{\leqslant}
\renewcommand{\geq}{\geqslant}

\newcommand{\la}{\lambda}

\newcommand{\al}{\alpha}

\newcommand{\ep}{\varepsilon}

\newcommand{\x}{\mathbf{x}}

\newcommand{\dt}{\,dt}
\newcommand{\dx}{\,dx}

\newcommand{\lap}{\Delta}
\newcommand{\na}{\nabla}

\newcommand{\f}[2]{\frac{#1}{#2}}

\newcommand{\deb}{\rightharpoonup}

\title[Robustness of NLSE ground states with respect to point defect] {Existence, Structure, and robustness of ground states of a NLSE in 3D with a point defect} 

\author[R. Adami]{Riccardo Adami}
\address{Politecnico di Torino, Dipartimento di Scienze Matematiche ``G.L. Lagrange'',Corso Duca degli Abruzzi, 24, 10129, Torino, Italy}
\email{riccardo.adami@polito.it}

\author[F. Boni]{Filippo Boni}
\address{Università degli Studi di Napoli Federico II,Dipartimento di Matematica e Applicazioni ``Renato Caccioppoli”, Via Cintia, Monte S. Angelo, 80126, Napoli, Italy}
\email{filippo.boni@polito.it}

\author[R. Carlone]{Raffaele Carlone}
\address{Università degli Studi di Napoli Federico II,Dipartimento di Matematica e Applicazioni ``Renato Caccioppoli”, Via Cintia, Monte S. Angelo, 80126, Napoli, Italy}
\email{raffaele.carlone@unina.it}

\author[L. Tentarelli]{Lorenzo Tentarelli}
\address{Politecnico di Torino, Dipartimento di Scienze Matematiche ``G.L. Lagrange'',Corso Duca degli Abruzzi, 24, 10129, Torino, Italy}
\email{lorenzo.tentarelli@polito.it}

\date{\today}

\begin{document}


\begin{abstract}
We study the ground states for the Schr\"odinger equation with a focusing nonlinearity and a point interaction in dimension three. We establish that ground states exist for every value of the mass; moreover they are positive, radially symmetric, decreasing along the radial direction, and show a Coulombian singularity at the location of the point interaction.
Remarkably, the existence of the ground states is independent of the attractive or repulsive character of the point interaction.
\end{abstract}

\maketitle

\vspace{-.5cm}


\noindent {\footnotesize \textul{Keywords:} standing waves, nonlinear Schr\"odinger, ground states, delta interaction, radially symmetric solutions, rearrangements.}



\section{Introduction}

\noindent
The standard Nonlinear Schr\"odinger Equation (NLSE) perturbed by a point interaction
\begin{equation}
 \label{eq-tNLS_form}
 i\f{\partial\psi}{\partial t}=(-\Delta+\alpha\delta_0)\psi \pm |\psi|^{p-2}\psi,\qquad\alpha\neq 0,\quad p>2,
\end{equation}
has been recently proposed as an effective model for a Bose-Einstein Condensate (BEC) in the presence of defects or impurities (see e.g. \cite{CC-94,SM-20,SCMS-20}).

While in dimension one the delta interaction is a bounded perturbation of the Laplacian in the sense of the quadratic forms and the corresponding solutions are widely studied (see e.g. \cite{ABD-20,ANR-20,AN-09, AN-CMP13,ANV-DCDS13,BD-21,LFFKS-08}), the analogous problem in higher dimension has been addressed only recently. In particular, well-posedness has been studied in dimensions two and three (\cite{CFN-21}), whereas properties of the standing waves have been investigated  in dimension two (\cite{ABCT-2d,georgiev}) only.

Here we extend the results obtained in \cite{ABCT-2d} to the three-dimensional setting. In particular, we establish the existence and some qualitative properties of the ground states, and show that such features are insensitive of the sign of the parameter $\alpha$. This is in contrast with the case of a particle subject to a point interaction in the absence of a nonlinearity, for which ground states exist only for negative $\alpha$.
This contrast persists even if the point interaction itself bears a nonlinearity, namely it is of the form $\alpha | \psi |^{p-2} \delta_0$, with $2 < p < 4$
(see \cite{ACCT-20,ACCT-21,ADFT-ANIHPC03,ADFT-ANIHPC04,AFH-21,ANO-JMP13,ANO-DCDS16,AT-JFA01,CCT-ANIHPC19,CFT-Non19,HL-20,HL-21}).

\subsection{Setting and main results}
\label{subsec-main}

Here we treat equation \eqref{eq-tNLS_form} in $\R^3$ in the \emph{focusing} case. Like in dimension two (\cite{ABCT-2d}), in three dimensions equation \eqref{eq-tNLS_form} is formal since the delta interaction is not a bounded perturbation of the Laplacian. The operator $-\Delta+\alpha\delta_0$ is then constructed through the theory of self-adjoint extensions of hermitian operators, which guarantees (see e.g. \cite{AGHKH-88}) the existence of a family $(H_\alpha)_{\alpha\in\R}$ of self-adjoint operators that realize all point perturbations of $-\Delta$. As a result, the domain of $H_\alpha$ is
\begin{equation*} \begin{split}
D (H_\alpha) \ = & \ \{ v \in L^2 (\Rd) : \exists q \in \C \ {\rm s.t.} \\ & \ v - \frac q {4 \pi |x|} = \phi \in \dot H^2 (\Rd) , \
 \nabla \phi \in H^1 (\Rd) \ {\rm and} \ \phi (0) = \alpha q \}
\end{split}
\end{equation*}
and its action reads
\begin{equation*}
H_\alpha v = - \Delta \phi.
\end{equation*}
The complex number $q$ is called the {\em charge} of the function $v$, and represents the size of the Coulombian singularity at the origin.

It proves convenient to represent the domain of $H_\alpha$ in an alternative way. One chooses an arbitrary positive number $\lambda$ and, denoted by $\G_\la$ the Green's function of $-\lap+\la$, namely
\begin{equation}
 \label{eq-green}
 \G_\la(\x):=\f{e^{-\sqrt{\la}|x|}}{4\pi|x|},
\end{equation}
one gets
\begin{multline}
\label{domop}
D(H_{\alpha}):=
\bigg\{v\in L^{2}(\Rd):\exists q\in\C \, :\text{ s.t. }\: \\ v-q\G_\la=:\phi_{\la}\in H^{2}(\Rd)\:\text{ and }\: \phi_{\la}(0)=\bigg(\alpha+\f{\sqrt{\la}}{4\pi}\bigg)q\bigg\},
\end{multline}
and
\begin{equation}
\label{eq-actH}
H_{\alpha}v:=-\lap\phi_{\la}-q\la\G_\la.
\end{equation}

Note that $\G_{\la}$ is not in $H^1 (\Rd)$ and belongs to $L^{p}(\Rd)$ if and only if $1\leq p<3$. This entails  $D(H_{\alpha})\subset L^{p}(\Rd)$ only if $2\leq p<3$, which is one of the major differences with the two-dimensional case where the embedding holds for $p > 2$.

In addition, one can see that functions in $D(H_{\alpha})$ consist of a \emph{regular part} $\phi_\la$, on which the operator $H_\alpha$ acts as the standard Laplacian, and a \emph{singular part} $q\G_\la$, on which the operator acts as the multiplication by $-\lambda$. The two components are connected by the {boundary condition} $\phi_{\la}(0)=\left(\alpha+\f{\sqrt{\la}}{4\pi}\right)q$. 
We stress that $\la$ is a dumb parameter that does not affect the definition of $H_\alpha$, since for every $\lambda > 0$ any function in $D(H_{\alpha})$ can be equivalently decomposed in regular and singular part.

The quadratic form associated with $H_\alpha$ has domain
\begin{equation}
\label{dom}
D:=\big\{v\in L^{2}(\Rd):\exists q\in\C,\,\lambda>0\:\text{ s.t. }\:v-q\G_\la=:\phi_{\la}\in H^{1}(\Rd)\big\},
\end{equation}
and action
\begin{equation}
\label{eq-Q}
Q(v):=\left\langle H_{\alpha}v,v\right\rangle=\|\na\phi_\la\|^2_{2}+\la\big(\|\phi_\la\|^2_{2}-\|v\|^2_{2}\big)+\bigg(\alpha+\f{\sqrt{\la}}{4\pi}\bigg)|q|^2,\qquad\forall v\in D,
\end{equation}
where we denoted by $\left\langle \cdot,\cdot\right\rangle$ the hermitian product in $L^2(\Rd)$ and by $\|\cdot\|_p$ the usual norm in  $L^p(\Rd)$. The value of $Q (v)$ is independent of the choice of $\lambda$. 
Notice that in the form domain no boundary condition is prescribed.

Finally, we denote by $-\omega_{\alpha}$ the bottom of the spectrum of $H_{\alpha}$, so
\begin{equation}
\label{bott-spect}
- \omega_{\alpha}:= \inf_{v\in D\setminus\{0\}}\f{Q(v)}{\|v\|_{2}^{2}}
\ = \ \left\{
\begin{array}{ll}
 -(4\pi\alpha)^{2}, & \text{if }\alpha<0,\\
 0,                 & \text{if }\alpha\geq0.
\end{array}
\right.
\end{equation}
Therefore the continuous spectrum of $H_\alpha$ is  $[0,+\infty)$, while the point spectrum of $H_{\alpha}$ is empty if $\alpha\geq 0$ and has the sole negative eigenvalue $-(4\pi\alpha)^{2}$ if $\alpha<0$.

We are ready to introduce the rigorous version of equation \eqref{eq-tNLS_form}, namely
\begin{equation}
\label{tdNLS}
i\f{\partial \psi}{\partial t}=H_\alpha \psi-|\psi|^{p-2}\psi,\qquad\alpha\in\R,\quad 2<p<3.
\end{equation}
Through the paper we shall refer to such equation as to
$\delta$-NLSE. It is well-known (\cite{CFN-21}) that its flow shows two conservation laws: mass, i.e. $L^2$-norm, and the
{\em energy}
\begin{equation}
\label{E}
E(v):=  \f{1}{2}Q(v)-\f{1}{p}\|v\|_p^p,
\end{equation}
defined for any $v$ in the form domain $D$.


Hereafter we  focus on the ground states of equation \eqref{tdNLS}, according to the following definition.

\begin{definition}

Given $\alpha\in \R$, $2<p<3$, $\mu>0$, we call {\em ground state at mass $\mu$} every global minimizer of the {$\delta$-NLS energy} functional defined in \eqref{E},
among the functions belonging to
\begin{equation*}
D_{\mu}:=\{v\in D:\|v\|_{2}^{2}=\mu\}.
\end{equation*}
\end{definition}
The main result of the paper is the following

\begin{theorem}[$\delta$-NLS ground states]
\label{exchar-gs}
Let $p\in(2,3)$ and $\alpha\in \R$. Therefore, for every $\mu>0$,
\begin{itemize}
 \item[(i)] there exists a ground state for the $\delta$-NLS at mass $\mu$;
 \item[(ii)] if $u=\phi_\la+q\G_\la$ is a  ground state, then:
 \begin{itemize}
  \item[(a)] there does not exist $\lambda>0$ such that $\phi_\la$ or $q\G_\la$ are identically zero;
  \item[(b)] $u$ is positive, radially symmetric, and decreasing along the radial direction, up to multiplication by a constant phase; moreover, $\phi_\la$ is nonnegative if $\lambda=\omega :=\mu^{-1}(\|u\|_p^p-Q(u)) $, and positive if $\lambda>\omega$.
 \end{itemize}
 \end{itemize}
\end{theorem}

\noindent
Ground states are particular cases of {\em bound states}, i.e. functions that satisfy
\begin{gather}
 \label{eq-regbound} u\in D(H_\alpha),\\
 \label{EL3} H_{\alpha}u+\omega u-|u|^{p-2}u=0.
\end{gather}
Bound states are the spatial profiles of {\em standing waves},
in the sense that $u$ is a bound state if and only if
$\psi(t,\x)=e^{i\omega t}u(\x)$, is a solution to equation
\eqref{tdNLS}.

In order to prove the qualitative features $(iia)$ and $(iib)$ 
in Theorem \ref{exchar-gs}, we will study a class of bound states larger than that of ground states, namely the set of the minimizers of the action functional $S_\omega$ defined as
 \begin{equation}
\label{SwE}
S_{\omega}: D\to \R\qquad\text{such that}\qquad S_\omega(v):=E(v)+\f{\omega}{2}\|v\|_2^2,
\end{equation} 
among the functions belonging to the Nehari's manifold
\begin{equation}
\label{eq-nehari}
N_{\omega}:=\{v\in D\setminus\{0\}:I_{\omega}(v)=0\},
\end{equation}
where $I_\omega:D\to\R$ is given by
\begin{equation}
\label{eq-Inehari}
I_{\omega}(v):=\langle S_{\omega}'(v),v\rangle=\|\na\phi_\la\|^2_{2}+\la\|\phi_\la\|^2_{2}+(\omega-\la)\|v\|^2_{2}+\bigg(\alpha+\f{\sqrt{\la}}{4\pi}\bigg)|q|^2-\|v\|^{p}_{p}.
\end{equation}



The result on the minimizers of the action functional reads as follows.

\begin{theorem}[$\delta$-NLS action minimizers]
\label{exchar-actmin}
Let $p\in(2,3)$, $\alpha\in \R$ 
Then, 
\begin{itemize}
\item[(i)] a minimizer of the action of the $\delta$-NLS at frequency $\omega$ does exist if and only if $\omega> {\omega}_{\alpha}$ (defined in \eqref{bott-spect});

\item[(ii)] if $u=\phi_\la+q\G_\la$ is a minimizer of the action of the $\delta$-NLS at frequency $\omega>{\omega}_{\alpha}$, then:
\begin{itemize}
 \item[(a)] there does not exist $\lambda>0$ such that $\phi_\la$ or $q\G_\la$ are identically zero;
 \item[(b)] $u$ is positive, radially symmetric, and decreasing along the radial direction, up to multiplication by a constant phase factor; in particular, $\phi_\la$ is nonnegative when $\lambda=\omega$, and positive when $\lambda>\omega$.
\end{itemize}
\end{itemize}
\end{theorem}

Through the following Lemma (whose proof can be found in  \cite[Appendix B]{ABCT-2d} and is an adaptation of what has been established in \cite{DST-21} for the standard NLSE) we can connect minimizers of the action with ground states.

\begin{lemma}
\label{chargs}
Let $p\in(2,3)$, $\alpha\in\R$ and $\mu>0$. If $u$ is a ground state of mass $\mu$, then it is also a minimizer of the action at the frequency $\omega=\mu^{-1}(\|u\|_p^p-Q(u))$.
\end{lemma}

In view of this and of Lemma \ref{chargs}, one can  see that point (ii) of Theorem \ref{exchar-gs} is a straightforward consequence of Theorem \ref{exchar-actmin}. We also mention that establishing (ii)(b) of Theorem \ref{exchar-actmin} requires an equivalent formulation of the problem of the minimization of the action. We shall in fact minimize
\[
 Q_\omega(v):=Q(v)+\omega\|v\|^2
\]
among the functions in $D$ with fixed $L^p$ norm. This technique is purely variational and does not retraces the classical ones used for the standard NLSE.

Finally, from Theorem \ref{exchar-gs} and Theorem \ref{exchar-actmin} it appears that the sign of $\alpha$ does not affect the behavior of the ground states and of the minimizers of the action. As mentioned at the beginning, this robustness is, at first sight, surprising. More precisely, while in dimension two this is natural as the sign of $\alpha$ does not even affect the existence of ground states for the linear problem, in dimension three ground  states of the linear problem exist only for negative values of $\alpha$. In other words, the intuitive idea that ground states are deformations of the linear ground state due to the ignition of a nonlinearity, is misleading. Such description is inspired by the fact that minimizers of the action exist if and only if the frequency exceeds the bottom of the spectrum of $H_{\alpha}$, which in dimension two coincides with its only eigenvalue. Yet in dimension three the analogy fails since when $\alpha\geq 0$ the eigenvalue disappears, but the minimizers of the action still exist if and only if the frequency exceeds the bottom of the spectrum of $H_{\alpha}$.

An intuitive explanation of this phenomenon can be drawn by describing the problem from another point of view. If one interprets the model as a delta perturbation of the NLSE, then one immediately sees that $D\supset H^{1}(\R^{3})$ and so, even perturbing the standard NLSE with a repulsive delta interaction, the infimum of the action gets lower with respect to the infimum of unperturbed action. Thus, the perturbed problem is in any case energetically convenient with respect to the standard one.

\smallskip
\textbf{Notation.} In the following, we use the expressions \emph{$\delta$-NLS ground states} and \emph{NLS ground states} to refer to the global minimizers of the $\delta$-NLS energy and the standard NLS energy, respectively. We use \emph{$\delta$-NLS action minimizers} and \emph{NLS action minimizers} in an analogous way.


\subsection*{Organization of the paper}

\begin{itemize}
 \item Section \ref{sec-prel} introduces some preliminary results that are useful throughout the paper; more precisely:
 \begin{itemize}
  \item in Section \ref{subsec-green} we recall some well-known features of the Green's function of $-\Delta+\lambda$,
  \item in Section \ref{subsec-GN} we establish two extensions of the Gagliardo-Nirenberg inequality (Proposition \ref{GNsing}),
 \end{itemize}
 \item Section \ref{sec-gs} addresses the existence of  ground states (Theorem \ref{exchar-gs}--(i));
 \item Section \ref{sec-am} addresses the existence of action minimizers (Theorem \ref{exchar-actmin}--(i));
 \item Section \ref{sec-char} establishes the main features both of the ground states and of the action minimizers (Theorem \ref{exchar-gs} --(ii)/Theorem \ref{exchar-actmin}--(ii)).
\end{itemize}




\section{Preliminary results}
\label{sec-prel}

In this section we collect some preliminary results, that will be exploited in the proofs of Theorem \ref{exchar-gs} and Theorem \ref{exchar-actmin}.


\subsection{Properties of the Green's function}
\label{subsec-green}

First, by \eqref{eq-green} one can easily check that $\G_\la\in L^r(\Rd)$ for every $r\in[1,3)$, with
\begin{equation}
\label{eq-GLP}
 \|\G_{\la}\|_{2}^{2}=\f{1}{8\pi\sqrt{\la}}\qquad\text{and}\qquad \|\G_{\la}\|_{r}^{r}= \f{\|\G_{1}\|_r^r}{\la^{\f{3-r}{2}}},\quad\text{when}\quad r\in[1,3),
\end{equation}
and that
\begin{equation}
\label{eq-Glanu}
\|\G_{\la}-\G_{\nu}\|_{2}^{2}=\f{1}{8\pi}\left(\f{1}{\sqrt{\la}}+\f{1}{\sqrt{\nu}}-\f{4}{\sqrt{\la}+\sqrt{\nu}}\right).
\end{equation}
We recall that $\G_\la$ is positive, radially symmetric, decreasing along the radial direction, exponentially decaying at infinity and smooth up to the origin. Moreover, $\G_\la\not\in H^1(\Rd)$ and
\begin{equation}
\label{eq-DGlanu}\|\na(\G_{\la}-\G_{\nu})\|_{2}^{2}=\f{1}{8\pi}\left(\f{3\la\sqrt{\nu}-3\nu\sqrt{\la}+\nu\sqrt{\nu}-\la\sqrt{\la}}{\nu-\la}\right).
\end{equation}
Finally, whenever $\nu<\lambda$,
\begin{equation}
\label{Gla<Gnu}
\G_\la(x)=\f{e^{-\sqrt{\la}|x|}}{4\pi|x|}<\f{e^{-\sqrt{\nu}|x|}}{4\pi|x|}=\G_\nu(x),\qquad\forall x\in\Rd\setminus\{0\}.
\end{equation}


\subsection{Gagliardo-Nirenberg inequalities}
\label{subsec-GN}

Here we aim at finding a version of Gagliardo-Nirenberg inequality for the energy space $D$, defined by \eqref{dom}. 

First, recall the standard three-dimensional Gagliardo-Nirenberg inequality (see e.g. \cite[Theorem 1.3.7]{C-CL03}): for every $p\in (2,6)$ there exists $C_{p}>0$ such that
\begin{equation}
\label{GNclass}
\|v\|_{p}^{p}\leq C_{p}\|\na v\|_{2}^{\f{3(p-2)}{2}}\|v\|_{2}^{\f{6-p}{2}}, \qquad\forall\, v \in H^{1}(\Rd).
\end{equation}


Second, every function with $q \neq 0$ can be decomposed into a regular and a singular part according to the choice $\lambda = \ep\f{|q|^{4}}{\|u\|_{2}^{4}}$, with $\ep >0$ arbitrarily chosen, i.e.
\begin{equation}
\label{Sing}
u =  \phi + q\G_{\ep\f{|q|^{4}}{\|u\|_{2}^{4}}}, \quad \phi \in H^1 (\Rd)
\end{equation} Using such decomposition, one can prove the following result.

 \begin{proposition}[Gagliardo-Nirenberg inequalities]
 \label{GNsing}
 For every $p\in(2,3)$, there exists $K_{p}>0$ such that
 \begin{equation}
 \label{GND}
 \|v\|_p^p\leq K_{p}\left(\|\na \phi_\la\|_{2}^{\f{3(p-2)}{2}}\|\phi_\la\|_{2}^{\f{6-p}{2}}+\f{|q|^p}{\la^{\f{3-p}{2}}}\right), \qquad\forall v=\phi_\la+q\G_{\la} \in D,\quad\forall\la>0.
 \end{equation}
 Moreover, there exists $M_{p,\ep}>0$ such that
\begin{multline}
\label{GNDs}
\|v\|_{p}^{p}\leq M_{p,\ep}\left(\|\na \phi\|_{2}^{\f{3(p-2)}{2}}\|v\|_{2}^{\f{6-p}{2}}+|q|^{3(p-2)}\|v\|_{2}^{2(3-p)}\right),\\[.2cm]
\forall v=\phi+q\G_{\ep\f{|q|^{4}}{\|v\|_{2}^{4}}} \in D\setminus H^1(\Rd).
\end{multline}

\end{proposition}
\begin{proof}
If we fix $v=\phi_\la+q\G_\la\in D$, for some $\la>0$, then \eqref{GNclass} and \eqref{eq-GLP} yield
\begin{equation*}
\|v\|_{p}^{p}=\|\phi_{\la}+q\G_{\la}\|_{p}^{p}\leq 2^{p-1}\left(\|\phi_{\la}\|_{p}^{p}+|q|^{p}\|\G_{\la}\|_{p}^{p}\right)\leq K_{p}\left(\|\na \phi_\la\|_{2}^{\f{3(p-2)}{2}}\|\phi_\la\|_{2}^{\f{6-p}{2}}+\f{|q|^p}{\la^{\f{3-p}{2}}}\right),
\end{equation*}
that is \eqref{GND}. On the other hand, if we also assume that $q\neq 0$ and set $\la=\la_{q,\ep}:=\ep\f{|q|^{4}}{\|v\|_{2}^{4}}$, form some $\ep>0$, then by \eqref{eq-GLP}, \eqref{GND} and the triangle inequality there results
\begin{equation*}
\begin{split}
\|v\|_{p}^{p}&\leq 2^{\f{p-2}{2}}K_{p}\left(\|\na \phi\|_{2}^{\f{3(p-2)}{2}}\|v\|_{2}^{\f{6-p}{2}}+\|\na \phi\|_{2}^{\f{3(p-2)}{2}}\f{|q|^{\f{6-p}{2}}}{\la_{q,\ep}^{\f{6-p}{8}}}+\f{|q|^{p}}{\la_{q,\ep}^{\f{3-p}{2}}}\right)\\[.2cm]
&\leq M_{p,\ep}\left(\|\na \phi\|_{2}^{\f{3(p-2)}{2}}\|v\|_{2}^{\f{6-p}{2}}+|q|^{3(p-2)}\|v\|_{2}^{2(3-p)}\right),
\end{split}
\end{equation*}
which concludes the proof.
\end{proof}

\begin{remark}
Note that, whenever $q=0$, i.e. $v\in H^1(\Rd)$, \eqref{GND} reduces to \eqref{GNclass}. Note also that, in contrast to the standard Gagliardo-Nirenberg inequality, we must limit ourselves to the powers $p<3$ since $\G_\la\not\in L^p(\Rd)$ when $p\geq 3$. Finally, we highlight that $M_{p,\ep}:=2^{\f{p-2}{2}}K_p\max\big\{1,\ep^{\f{p-p}{8}},\ep^{\f{p-3}{2}}\big\}$ and, thus, $M_{p,\ep}\to+\infty$, as $\ep\downarrow0$.
\end{remark}



\section{Existence of ground states}
\label{sec-gs}

Here we prove point (i) of Theorem \ref{exchar-gs}, which is the existence of  ground states of mass $\mu$ for every $\mu>0$. To this aim, some further notation is required: we denote by $\Eps(\mu)$ the $\delta$-NLS energy infimum at mass $\mu$, i.e.
\[
 \Eps(\mu):=\inf_{v\in D_\mu}E(v),
\]
with $E$ defined in \eqref{E}, and by $\Eps^{0}(\mu)$ the NLS energy infimum at mass $\mu$, i.e.
\[
 \Eps^{0}(\mu):=\inf_{v\in H_\mu^1(\Rd)}E^0(v),
\]
where
\[
 E^0(v):=\frac{1}{2}\|\nabla v\|_2^2-\frac{1}{p}\|v\|_p^p\qquad\text{and}\qquad H_\mu^1(\Rd):=\{v\in H^1(\Rd):\|v\|_2^2=\mu \}.
\]

As a preliminary step we establish boundedness from below of $E$ restricted to $D_\mu$, whenever $p\in(2,3)$. By the decomposition introduced in \eqref{Sing}, the functional $E$ reads
\begin{multline}
\label{newE1}
E(u)=\\[.0cm]
\left\{
\begin{array}{ll}
\displaystyle \f{1}{2}\|\na\phi\|^2_{2}+\f{\ep|q|^{4}\|\phi\|_{2}^{2}}{2\|u\|_{2}^{4}}+\f{\alpha|q|^{2}}{2}+\f{|q|^{4}}{2\|u\|_{2}^{2}}\left(\f{\sqrt{\ep}}{4\pi}-\ep\right)-\f{\|u\|^{p}_{p}}{p}, &\text{if } u\in  D\setminus H^1(\Rd),\\[.6cm]
\displaystyle \f{1}{2}\|\na u\|^2_{2}-\f{1}{p}\|u\|^{p}_{p}, &\text{if } u \in H^{1}(\Rd).
\end{array}
\right.
\end{multline}
In addition, if one fixes $\ep<\f{1}{16\pi^{2}}$, then the coefficient in front of $\f{|q|^{4}}{2\|u\|_{2}^{2}}$ is positive. For instance,  we choose $\ep=\f{1}{64\pi^{2}}$, so that 
\begin{equation}
\label{Epala}
E(u)= \f{1}{2}\|\na\phi\|^2_{2}+\f{|q|^{4}\|\phi\|_{2}^{2}}{128\pi^2\|u\|_{2}^{4}}+\f{\alpha|q|^{2}}{2}+\f{|q|^{4}}{128\pi^{2}\|u\|_{2}^{2}}-\f{\|u\|^{p}_{p}}{p},\qquad \forall u\in D\setminus H^{1}(\R^{3}).
\end{equation}

\begin{proposition}
\label{Epabound}
For any fixed $p\in(2,3)$ and $\alpha\in\R$,
\[
 \Eps(\mu)>-\infty,\qquad\forall\mu>0.
\]
\end{proposition}

\begin{proof}
Let $u\in D_\mu$. We manage separately the cases $u\in H_\mu^{1}(\Rd)$ and $u\in D_\mu\setminus H_\mu^1(\Rd)$. In the former case, combining \eqref{newE1} and \eqref{GNclass}, there results
\begin{equation*}
E(u)\geq \f{1}{2}\|\na u\|^2_{2}-\f{C_{p}}{p}\|\na u\|^{\f{3(p-2)}{2}}_{2}\mu^{\f{6-p}{4}},
\end{equation*}
and thus $E$ is bounded from below on $H_\mu^{1}(\Rd)$ as $p\in(2,3)$ by assumption. In the latter case, combining \eqref{GNDs} and \eqref{Epala} and denoting by $M_p$ the constant $M_{p,\ep}$ for $\ep=\f{1}{64\pi^{2}}$, there results
\begin{multline}
\label{Eun-prel}
E(u)\geq \f{\|\na\phi\|^2_{2}}{2}-\f{M_{p}\|\na\phi\|^{\f{3(p-2)}{2}}_{2}\mu^{\f{6-p}{4}}}{p}+\f{|q|^{4}\|\phi\|^2_{2}}{128\pi^2\mu^{2}}\\
+\f{\alpha|q|^{2}}{2}+\f{|q|^{4}}{128\pi^{2}\mu}-\f{M_{p}|q|^{3(p-2)}\mu^{3-p}}{p},
\end{multline}
and thus $E$ is bounded from below also on $D_\mu\setminus H_\mu^1(\Rd)$ as, again, $p\in(2,3)$ by assumption.
\end{proof}

In the following proposition we compare the infima of the energy of the $\delta$-NLS  and of the NLS. 

\begin{proposition}
\label{Eps<Eps0}
For any fixed $p\in(2,3)$ and $\alpha\in\R$,
\begin{equation}
\label{eq-levcomp}
\Eps(\mu)<\Eps^{0}(\mu)<0,\qquad\forall \mu>0.
\end{equation} 
\end{proposition}

The proof of such proposition requires the following well-known result about the NLS ground states (see e.g. \cite{L-ANIHPC84} and \cite{GNN-79}).

\begin{proposition}
\label{exstanden}
Let $p\in\left(2,\f{10}{3}\right)$. There exists an NLS ground state of mass $\mu$ for every $\mu>0$. In addition, such minimizer is unique, positive, radially symmetric and decreasing along the radial direction, up to  multiplication by a constant phase and translations.
\end{proposition}

Throughout the paper, we denote the positive symmetric NLS ground state of mass $\mu$, usually called {\em soliton}, by $S_\mu$.

\begin{proof}[Proof of Proposition \ref{Eps<Eps0}]
For any $\mu>0$, $S_\mu$ cannot be a $\delta$-NLS ground state of mass $\mu$. Indeed, a $\delta$-NLS ground state has to satisfy the boundary condition in \eqref{eq-regbound}, namely $\phi_\la(0)=(\alpha+\f{\sqrt{\la}}{4\pi})q$; but, since $S_\mu\in H^1(\Rd)$, $q=0$ and $S_\mu\equiv\phi_\la$ so that $S_\mu(0)=0$, which contradicts the positivity of $S_\mu$. As a consequence, there must exists $v\in D_\mu$ such that $E(v)<E(S_\mu)=\Eps^{0}(\mu)$, which concludes the proof of the former inequality in \eqref{eq-levcomp}.

Concerning the latter inequality, fix $\mu>0$ and consider $v\in H^{1}_{\mu}(\Rd)$. Using the mass-preserving transformation
 \begin{equation*} 
v_{\sigma}(x)=\sigma^{\f{3}{2}} v(\sigma x),
\end{equation*}
one obtains
\begin{equation*}
E^0(v_{\sigma})=\f{\sigma^{2}}{2}\|\na v\|_{2}^{2}-\f{\sigma^{\f{3(p-2)}{2}}}{p}\|v\|_{p}^{p}
\end{equation*}
and thus, since $p\in(2,3)$, choosing a small enough $\sigma$ one gets $\Eps^0(\mu)\leq  E^{0}(v_{\sigma})<0$, and the proof is complete.
\end{proof}

The last two preliminary tools necessary for the proof of point (i) of Theorem \ref{exchar-gs} are provided by the next two lemmas and concern the minimizing sequences at mass $\mu$ of the {$\delta$-NLS energy}.

\begin{lemma} \label{q>C}
\label{minimseq1}
Let $p\in(2,3)$ and $\alpha\in\R$. For every minimizing sequence 
$u_n =\phi_{\lambda,n} + q_n \G_\lambda$
of the $\delta$-NLS energy at mass $\mu$, there exist $\bar{n}\in \N$ and $C > 0$, such that
\[
 |q_n| > C,\qquad\forall n\geq \bar{n}.
\]
\end{lemma}

\begin{proof}
Assume by contradiction that there exists a minimizing sequence for the $\delta$-NLS energy at mass $\mu$ such that $q_n \to 0$. Then, $\| \phi_{\lambda,n} \|_2^2$ is bounded since it converges to $\mu$.
Moreover, combining \eqref{E} and \eqref{GND}, one obtains
\begin{equation*}
\begin{split}
E (u_n) &\geq  \f 1 2 \| \nabla \phi_{\lambda,n} \|_2^2 + \f \lambda 2 ( \| \phi_{\lambda,n} \|_2^2 - \mu ) 
+ \f{( \alpha + \f{\sqrt{\la}}{4\pi})} 2 |q_n |^2 \\
&\quad  - \f {K_p} p \left(\|\nabla \phi_{\lambda,n} \|_2^{\f{3(p-2)}{2}} \| \phi_{\lambda,n}\|_2^{\f{6-p}{2}}+\f{|q_{n}|^{p}}{\la^{\f{3-p}{2}}}\right)\\
&=\f 1 2 \| \nabla \phi_{\lambda,n} \|_2^2- \f {K_p} p \|\nabla \phi_{\lambda,n} \|_2^{\f{3(p-2)}{2}} \| \phi_{\lambda,n}\|_2^{\f{6-p}{2}}+o(1),\qquad\text{as}\quad n \to +\infty.
\end{split}
\end{equation*}
Thus, as $E(u_n)$ is bounded from above and $p<3$, $\| \nabla \phi_{\lambda,n} \|_{2}$ is bounded.

Now, define the sequence $ \xi_n := \frac {\sqrt \mu}{\| \phi_{\lambda,n} \|_2} \phi_{\lambda,n}$, so that $\| \xi_n \|_2^2 = \mu$, for every $n\in N$, and 
$ \| \nabla \xi_n \|_2^2 = \frac \mu  {\| \phi_{\lambda,n} \|^2_2} \| \nabla \phi_{\lambda,n} \|^2_2 $ is bounded. Then, as $\phi_{\lambda,n} - u_n \to 0$ in $L^r (\Rd)$, for all $r\in[2,3)$, using the properties of $\xi_n$, we find that
\begin{eqnarray*}
E (u_n) & = & E^0 (\phi_{\lambda, n}) + o (1) \ = \ E^0 (\xi_n) + o (1) \\
& \geq & E^0 (S_\mu) + o (1),\qquad\text{as}\quad n \to +\infty,
\end{eqnarray*}
where $S_\mu$ denotes again the soliton of mass $\mu$. Hence, passing to the limit,
$$ \mathcal E (\mu) \ \geq \mathcal E^0 (\mu), $$
which contradicts \eqref{eq-levcomp}, thus implying that $q_n\not\to0$. Since this is true for every subsequence of any minimizing sequence of the $\delta$-NLS energy, this concludes the proof.
\end{proof}

\begin{lemma}
\label{minimseq2}
Let $p\in(2,3)$, $\alpha\in\R$, $\mu>0$, $(u_{n})_n$  a minimizing sequence of the $\delta$-NLS energy  at mass $\mu$. Then, $(u_{n})_n$ is bounded in $L^r(\Rd)$ for every $r\in[2,3)$, and there exists $u\in D\setminus H^1(\Rd) $ such that, up to subsequences,
\begin{itemize}
 \item $u_n\deb u$ in $L^2(\Rd)$,
 \item $u_n\to u$ a.e. in $\Rd$.
\end{itemize}
Moreover, if one fixes $\lambda>0$ and considers the decomposition $u_{n}=\phi_{n,\la}+q_{n}\G_{\la}$, then $(\phi_{n,\la})_n$ and $(q_n)_n$ are bounded in $H^1(\Rd)$ and $\C$, respectively, and there exist $\phi_{\la}\in H^{1}(\Rd)$ and $q\in \C\setminus\{0\}$ such that $u=\phi_\la+q\G_\la$ such that up to subsequences,
\begin{itemize}
 \item $\phi_{n,\la}\deb\phi_{\la}$ in $L^2(\Rd)$,
 \item $\nabla\phi_{n,\la}\deb\nabla\phi_{\la}$ in $L^2(\Rd)$,
 \item $q_{n}\rightarrow q$ in $\C$,
\end{itemize}
as $n\to+\infty$.
\end{lemma}

\begin{proof}
The proof follows from classical arguments and retraces that of \cite[Lemma 3.5]{ABCT-2d}. We sketch it here  for the sake of completeness only.

By Banach-Alaoglu Theorem, $u_n\deb u$ in $L^2(\Rd)$ up to subsequences. Moreover, owing to Lemma \ref{q>C}, it is not restrictive to assume that $|q_n| > C > 0$, for every $n\in \N$. As a consequence, we can use the decomposition introduced by \eqref{Sing}, namely $u_{n}=\phi_{n}+q_{n}\G_{\nu_{n}}$ with $\la=\nu_{n}:=\f{1}{64\pi^{2}}\f{|q_{n}|^{4}}{\|u_{n}\|_{2}^{4}}$, and \eqref{GNDs}.

Now, arguing as in \cite[Lemma 3.5]{ABCT-2d} and using \eqref{GNDs} and \eqref{Epala}, one gets that $\phi_{n}$ is bounded in $H^{1}(\R^{3})$ and $q_{n}$ is bounded in $\C$. Finally, in order to prove the thesis for every $\la>0$ fixed, we argue again as in \cite[Lemma 3.5]{ABCT-2d} recalling that $\phi_{\la,n}:=\phi_{n}+q_{n}(\G_{\nu_{n}}-\G_{\la})$, using \eqref{eq-Glanu}, \eqref{eq-DGlanu} and that $q_n$ is bounded from above and away from zero, and distinguishing the cases $\lambda\geq1+(8\pi\mu)^{-2}\sup_n|q_n|^4$ and $\lambda<1+(8\pi\mu)^{-2}\sup_n|q_n|^4$. 
\end{proof}

Eventually, we can prove the existence of the $\delta$-NLS ground states.

\begin{proof}[Proof of Theorem \ref{exchar-gs}-(i)]
Let $(u_{n})_n$ be a minimizing sequence of the $\delta$-NLS energy at mass $\mu$. As we saw before, it is not restrictive to assume that $u_{n}=\phi_{n,\la}+q_{n}\G_{\la}$, with $q_{n}\neq 0$ and $\la>0$.  Hence Lemma \ref{minimseq2} applies. 

In order to conclude the proof we argue as follows (note that all the limits below have to be meant up to subsequences). Set $m:=\|u\|_2^2$. Weak lower semicontinuity of the $L^2(\Rd)$-norm implies $m\leq\mu$, while $q\neq 0$ implies $m\neq0$. Then suppose, by contradiction, that $m\in(0,\mu)$. Since $p>2$ and $\f{\mu}{\|u_n-u\|_2^2}>1$ for $n$ sufficiently large, there results that
\begin{equation}
\label{E-un-u}
\liminf_n E(u_n-u)\geq \f{\mu-m}{\mu}\Eps(\mu).
\end{equation}
On the other hand, it is possible to show in an analogous way that
\begin{equation}
\label{E-u}
E(u)>\f{m}{\mu}\Eps(\mu).
\end{equation}
Moreover, recalling that $u_n\deb u,\,\phi_{n,\la}\deb\phi_{\la},\,\na\phi_{n,\la}\deb\na\phi_{\la}$ in $L^2(\Rd)$ and $q_n\to q$ and using the Brezis-Lieb lemma (\cite{BL-83}), we have that
\begin{equation}
\label{BLext}
E(u_n)=E(u_n-u)+E(u)+o(1)\qquad\text{as} \quad n\to +\infty
\end{equation}
Combining \eqref{E-un-u}, \eqref{E-u} and \eqref{BLext}, one can see that
\begin{equation*}
\Eps(\mu)=\liminf_n E(u_n)=\liminf_n E(u_n-u)+ E(u)>\f{\mu-m}{\mu}\Eps(\mu)+\f{m}{\mu}\Eps(\mu)=\Eps(\mu),
\end{equation*}
which is a contradiction. Therefore, $m=\mu$, which means that $u\in D_\mu$ and that $u_n\to u$,  $\phi_{n,\la}\to \phi_\la$ in $L^2(\Rd)$. Finally, by \eqref{GND}, $u_n\to u$ in $L^p(\Rd)$ and thus
\begin{equation}
\label{eq-semicont}
E(u)\leq \liminf_n E(u_n)=\Eps(\mu),
\end{equation}
which completes the proof.
\end{proof}



\section{Existence of the minimizers of the action}
\label{sec-am}

In this section we prove (i) of Theorem \ref{exchar-actmin}, which is the existence/nonexistence of the minimizers of the $\delta$-NLS action at frequency $\omega$. It is convenient to introduce some notation. First, we denote by $d(\omega)$ the infimum of the $\delta$-NLS action at frequency $\omega$, i.e.
\[
 d(\omega):=\inf_{v\in N_\omega}S_\omega(v),
\]
where $S_\omega,\,N_\omega$ are given by \eqref{SwE} and \eqref{eq-nehari}, and define
\begin{equation}
\label{Qomega}
Q_\omega(v):=Q(v)+\omega\|v\|_2^2,
\end{equation} 
where $Q$ is given by \eqref{eq-Q}, so that
\begin{equation}
\label{eq-utile}
 S_\omega(v)=\f{1}{2}Q_\omega(v)-\f{1}{p}\|v\|_p^p\qquad\text{and}\qquad I_\omega(v)=Q_\omega(v)-\|v\|_p^p.
\end{equation}
On the other hand, we denote by $d^0(\omega)$ the infimum of the NLS action at frequency $\omega$, i.e.
\[
 d^0(\omega):=\inf_{v\in N^0_\omega}S_\omega^0(v),
\]
where
\begin{gather*}
 S_\omega^0(v):=E^0(v)+\f{\omega}{2}\|v\|_2^2,\\
 N_\omega^0:=\{v\in H^1(\Rd)\setminus\{0\}:I_\omega^0(v)=0\},\qquad I_\omega^0(v):=\langle {S_\omega^0}'(v),v\rangle.\\
\end{gather*}

Preliminarily, we note that
\begin{equation}
\label{eq-funceq}
 S_\omega(v)=\widetilde{S}(v)>0,\qquad\forall v\in N_\omega,\qquad\text{with}\qquad \widetilde{S}(v):=\f{p-2}{2p}\|v\|_p^p.
\end{equation}
As a consequence, since ${S_\omega}_{|_{H^1(\Rd)}}=S_\omega^0$ and $N_\omega\cap H^1(\Rd)=N_\omega^0$, there results 
\begin{equation}
\label{0dwd0w}
0\leq d(\omega)\leq d^0(\omega),\qquad\forall \omega\in\R.
\end{equation}
Furthermore, since $d^0(\omega)=0$ for every $\omega\leq 0$ (\cite[Lemma 2.4 and Remark 2.5]{DST-21}), it is straightforward that $d(\omega)=0$, for every $\omega\leq0$, so that there are no minimizers of the $\delta$-NLS action at $\omega\leq 0$. As a consequence, we only address  the case $\omega>0$. 
 
The former step of our proof is to investigate when inequalities \eqref{0dwd0w} are strict. We introduce the set
\begin{equation*}
\widehat{N}_{\omega}:=\{q\G_{\la}: \la>0,\,q\in\C\setminus\{0\},\,I_{\omega}(q\G_{\la})=0\},
\end{equation*}
which is the subset of $N_\omega$ of the functions that admit a representation with the sole singular part for a value of $\lambda>0$. We can characterize $\widehat{N}_{\omega}$ as follows.

\begin{lemma}
\label{qlacond}
Let $p\in(2,3)$, $\alpha\in\R$ and $\omega>0$. Then, $q\G_{\la}\in \widehat{N}_{\omega}$ if and only if
\begin{equation}
\label{admla}
\f{\omega}{\sqrt{\la}}+8\pi \alpha+\sqrt{\la}>0
\end{equation}
and
\begin{equation}
\label{qandla}
|q|=\f{\la^{\f{3-p}{2(p-2)}}}{\kappa_{p}}\left[\f{\omega}{\sqrt{\la}}+8\pi \alpha+\sqrt{\la}\right]^{\f{1}{p-2}},
\end{equation}
with $\kappa_p=(8\pi\|\G_1\|_p^{p})^{\f{1}{p-2}}$.

\end{lemma}
\begin{proof}
Let $q\neq 0$ and $\la>0$. By \eqref{eq-GLP}, $I_{\omega}(q\G_{\la})=0$ if and only if
\begin{equation*}
\f{\omega-\la}{8\pi\sqrt{\la}}|q|^{2}+\left(\alpha+\f{\sqrt{\la}}{4\pi}\right)|q|^{2}-\f{\kappa}{\la^{\f{3-p}{2}}}|q|^{p}=0, 
\end{equation*}
with $\kappa:=\|\G_1\|_p^p$, so that
\begin{equation}
|q|^{p-2}=\f{\la^{\f{3-p}{2}}}{\kappa}\left(\f{\omega-\la}{8\pi\sqrt{\la}}+\alpha+\f{\sqrt{\la}}{4\pi}\right)=\f{\la^{\f{3-p}{2}}}{8\pi \kappa}\left(\f{\omega}{\sqrt{\la}}+8\pi \alpha+\sqrt{\la}\right).
\end{equation}
Since $|q|^{p-2}>0$, \eqref{admla} and \eqref{qandla} follow.
\end{proof}

\begin{lemma}
\label{Nomega}
Let $p\in(2,3)$, $\alpha\in \R$ and $\omega>0$. Then
\begin{itemize}
 \item[(i)] if $\alpha<0$ and $\omega\in(0,{\omega}_{\alpha})$, then there exists $\la_{1}(\omega)\in(0,{\omega}_{\alpha})$ and $\la_{2}(\omega)>{\omega}_{\alpha}$ such that 
\begin{equation}
\widehat{N}_{\omega}=\{q\G_{\la}: \la\in(0,\la_{1}(\omega))\cup(\la_{2}(\omega),+\infty)\text{ and }q\in\C\setminus\{0\}\text{ and satisfies \eqref{qandla}}\};
\end{equation}
in particular, $\la_{1}(\omega)$ and $\la_{2}(\omega)$ are the sole solutions of the equation
\begin{equation*}
\f{\omega}{\sqrt{\la}}+8\pi \alpha+\sqrt{\la}=0;
\end{equation*}
 \item[(ii)] if $\alpha<0$ and $\omega={\omega}_{\alpha}$, then
\begin{equation}
\widehat{N}_{\omega}=\{q\G_{\la}:\la>0,\,\la\neq {\omega}_{\alpha},\text{ and }q\in\C\setminus\{0\}\text{ and satisfies \eqref{qandla}}\};
\end{equation}
\item[(iii)] if $\omega>{\omega}_{\alpha}$, then 
\begin{equation}
\widehat{N}_{\omega}=\{q\G_{\la}:\la>0\text{ and }q\in\C\setminus\{0\}\text{ and satisfies \eqref{qandla}}\}
\end{equation}
\end{itemize}
(We recall that $\omega_\al$ was defined in \eqref{bott-spect})
\end{lemma}

\begin{remark}
\label{rem-semp}
 We highlight that, in the previous result, the case $\alpha\geq0$ is always taken into account by (iii), since in this case ${\omega}_{\alpha}=0$.
\end{remark}

\begin{proof}[Proof of Lemma \ref{Nomega}]
Let $\omega>0$ and 
\begin{equation*}
g(\la):=\f{1}{\la^{\f{p-2}{2}}}\left(\la+8\pi\alpha\sqrt{\la}+\omega\right).
\end{equation*}
Recall that, in view of Lemma \ref{qlacond}, $q\G_{\la}\in \widehat{N}_{\omega}$ if and only if $g(\la)>0$ and $q$ satisfies \eqref{qandla}, namely $|q|=\kappa_p^{-1}g^{\f{1}{p-2}}(\la)$. First, we observe that, when $\alpha \geq 0$, ${\omega}_{\alpha}=0$ so that $g$ is strictly positive on $\R^+$, for any $\omega>{\omega}_{\alpha}$. Thus (iii) is straightforward in this case.

We focus then on $\alpha<0$.  One can easily check that, as $p\in(2,3)$,
\begin{equation*}
\lim_{\la\to 0^{+}}g(\la)=+\infty\,,\quad \lim_{\la\to +\infty} g(\la)=+\infty.
\end{equation*}
On the other hand, since
\[
 g'(\lambda)=\f{(4-p)\la+2(3-p)\pi\alpha\sqrt{\la}-(p-2)\omega}{2\la^{\f{p}{2}}},
\]
$g$ has a unique critical point on $\R^+$. Hence there exists $\la^{*}=\la^{*}(\omega)$ such that $g$ is strictly decreasing for $\la\leq \la^{*}$ and strictly increasing for $\la \geq \la^{*}$, being $g(\la^{*})$ the global minimum of $g$ on $\R^+$. In particular, if $\omega>{\omega}_{\alpha}$, then $g$ is strictly positive and \eqref{admla} admits a solution for every $\la>0$, so that (iii) is proved. On the contrary, if $\omega={\omega}_{\alpha}$, then $g$ vanishes for $\la= {\omega}_{\alpha}$ only, and \eqref{admla} admits a solution whenever $\la>0$ and $\la\neq{\omega}_{\alpha}$, so that (ii) is proved. Finally, if $\omega<{\omega}_{\alpha}$, then $g$ vanishes at two points $\la_{1,2}(\omega)=\left(\sqrt{{\omega}_{\alpha}}\mp\sqrt{{\omega}_{\alpha}- \omega}\right)^{2}$, with $\la_{1}(\omega)<{\omega}_{\alpha}<\la_{2}(\omega)$, and is negative for $\la\in [\la_{1}(\omega),\la_{2}(\omega)]$. Hence, \eqref{admla} does not admits any solution if and only if $\la\in [\la_{1}(\omega),\la_{2}(\omega)]$, so that (i) is proved.
\end{proof}

After this characterization of the set $\widehat{N}_\omega$, we can estimate the value of $d(\omega)$ for $\alpha<0$ and $\omega\in(0,{\omega}_{\alpha}]$ (note that, in view of Remark \ref{rem-semp} and the comments after \eqref{0dwd0w}, the analogous for $\alpha\geq0$ is trivial).

\begin{proposition}
\label{pro-dzero}
Let $p\in(2,3)$ and $\alpha<0$. Then, $d(\omega)=0$ for every  $\omega\in(0,{\omega}_{\alpha}]$. 
\end{proposition}

\begin{proof}
We give the proof in the cases $\omega\in(0,{\omega}_{\alpha})$ and $\omega={\omega}_{\alpha}$ separately. If $\omega\in(0,{\omega}_{\alpha})$, then by Lemma \ref{Nomega}
\begin{equation*}
\lim_{\substack{\la\to\la_{1}(\omega)^{-},\\ q\G_{\la}\in N_{\omega}}}|q|=\lim_{\la\to\la_{1}(\omega)^{-}}\f{\la^{\f{3-p}{2(p-2)}}}{\kappa_{p}}\left[\f{\omega}{\sqrt{\la}}+8\pi \alpha+\sqrt{\la}\right]^{\f{1}{p-2}}=0.
\end{equation*}
Hence, combining with \eqref{eq-funceq} and \eqref{eq-GLP},
\begin{multline*}
0\leq d(\omega)\leq \inf_{q\G_{\la}\in N_{\omega}}S_{\omega}(q\G_{\la})\leq \lim_{\substack{\la\to\la_{1}(\omega)^{-},\\ q\G_{\la}\in N_{\omega}}}S_{\omega}(q\G_{\la})\\
=\lim_{\substack{\la\to\la_{1}(\omega)^{-},\\ q\G_{\la}\in N_{\omega}}}\widetilde{S}(q\G_{\la})=\lim_{\substack{\la\to\la_{1}(\omega)^{-},\\ q\G_{\la}\in N_{\omega}}}\f{p-2}{2p}\|\G_{1}\|_{p}^{p}\f{|q|^{p}}{\la^{\f{3-p}{2}}}=0.
\end{multline*}
If, on the contrary, $\omega={\omega}_{\alpha}$, then one finds the same chain of inequalities and concludes by replacing the limits for $\la\to\la_{1}(\omega)^{-}$ with the limits for $\la\to{\omega}_{\alpha}$.
\end{proof}

The first consequence of this result if the following (again, the analogous for $\alpha\geq0$ is omitted since it is trivial).

\begin{corollary}
\label{noexistlowfreq}
Let $p\in(2,3)$ and $\alpha<0$. If $\omega\in(0,{\omega}_{\alpha}]$, then there does not exist any minimizer of the $\delta$-NLS action at frequency $\omega$.
\end{corollary}

\begin{proof}
The claim follows by Proposition \ref{pro-dzero} and \eqref{eq-funceq}.
\end{proof}

Before showing the proof of point (i) of Theorem \ref{exchar-actmin}, the last preliminary step is to discuss the behavior of $d(\omega)$ when $\omega>{\omega}_{\alpha}$. We start by recalling the following relation between $S_\omega$ and $\widetilde{S}$.

\begin{lemma}
\label{equivprob}
For every $p\in(2,3)$, $\alpha\in \R$, and $\omega>{\omega}_{\alpha}$
\begin{equation}
\label{eq-firstequiv}
d(\omega)=\inf_{v\in\widetilde{N}_\omega}\widetilde{S}(v),
\end{equation}
with
\[
 \widetilde{N}_\omega:=\{v\in D\setminus\{0\}:I_{\omega}(v)\leq0\}.
\]
In addition,
\begin{equation}
\label{eq-secondequiv}
\left\{\begin{array}{l}\displaystyle\widetilde{S}(u)=d(\omega)\\[.2cm]\displaystyle I_{\omega}(u)\leq 0\end{array}\right.\qquad\Longleftrightarrow\qquad \left\{\begin{array}{l}\displaystyle S_{\omega}(u)=d(\omega)\\[.2cm]I_{\omega}(u)=0.\end{array}\right.,
\qquad\forall u\in D\setminus\{0\}.
\end{equation}
\end{lemma}

\begin{proof}
 The proof follows from classical arguments and is analogous to that of \cite[Lemma 4.5]{ABCT-2d}.
\end{proof}

\begin{remark}
\label{rem-peq}
 Note that this result implies that, searching for minimizers of the $\delta$-NLS action  means searching for
\[
 u\in\widetilde{N}_\omega\qquad\text{such that}\qquad\widetilde{S}(u)=\inf_{v\in\widetilde{N}_\omega}\widetilde{S}(v)=d(\omega).
\]
\end{remark}

Now, the former point is to prove that the left inequality of \eqref{0dwd0w} is strict.

\begin{proposition}
\label{dpos}
For every $p\in(2,3)$, $\alpha\in\R$ and $\omega>{\omega}_{\alpha}$, there results that $d(\omega)>0$.
\end{proposition}

\begin{proof}
We can start by assuming $u\in \widetilde{N}_\omega\cap H^1(\Rd)$. By the Sobolev inequality
\begin{equation*}
0\geq I_{\omega}(u)=\|\na u\|_{2}^{2}+\omega\|u\|_{2}^{2}-\|u\|_{p}^{p}\geq C_p\|u\|_{p}^{2}+\omega\|u\|_{2}^{2}-\|u\|_{p}^{p}\geq C_p\|u\|_{p}^{2}-\|u\|_{p}^{p},
\end{equation*}
for some suitable $C_p>0$ only depending on $p$. Therefore, $\|u\|_{p}^{p-2}\geq C_p$, so that
\begin{equation}
\label{Spos}
\widetilde{S}(u)\geq \f{p-2}{2p} C_p^{\f{p}{p-2}},
\end{equation}
whence
\begin{equation}
\label{eq-infH1}
 \inf_{v\in\widetilde{N}_\omega\cap H^1(\Rd)}\widetilde{S}(v)\geq\f{p-2}{2p} C_p^{\f{p}{p-2}}>0.
\end{equation}
It is then left to study the case $u=\phi_\la+q\G_{\la}\in \widetilde{N}_\omega\setminus H^1(\Rd)$. Assume, without loss of generality, $\la\in({\omega}_{\alpha},\omega)$. Clearly $\alpha+\f{\sqrt{\la}}{4\pi}>0$, and thus there exists a constant $C>0$ such that
\begin{equation}
\label{Lpartest}
\|\na\phi_{\la}\|_{2}^{2}+\la\|\phi_{\la}\|_{2}^{2}+(\omega-\la)\|u\|_{2}^{2}+|q|^{2}\left(\alpha+\f{\sqrt{\la}}{4\pi}\right)\geq C\left(\|\phi_{\la}\|_{H^{1}}^{2}+|q|^{2}\right).
\end{equation}
In addition, by Sobolev inequality again, 
\begin{equation*}
\|u\|_{p}^{p}\leq C_{p}\left(\|\phi_{\la}\|_{p}^{p}+|q|^{p}\right)\leq C_{p}\left(\|\phi_{\la}\|_{H^1}^{p}+|q|^{p}\right)\leq C_{p}\left(\|\phi_{\la}\|_{H^1}^{2}+|q|^{2}\right)^{\f{p}{2}},
\end{equation*}
which implies
\begin{equation}
\label{up2}
\|\phi_{\la}\|_{H^1}^{2}+|q|^{2}\geq\f{1}{C_{p}}\|u\|_{p}^{2}.
\end{equation}
Then, combining \eqref{Lpartest} and \eqref{up2},
\begin{equation*}
0\geq I_{\omega}(u)\geq C\left(\|\phi_{\la}\|_{H^{1}}^{2}+|q|^{2}\right)-\|u\|_{p}^{p}\geq \f{C}{C_{p}}\|u\|_{p}^{2}-\|u\|_{p}^{p}
\end{equation*}
and thus, as before, there exists $K_p>0$, depending only on $p$, such that
\begin{equation*}
\widetilde{S}(u)\geq K_p,
\end{equation*}
whence
\begin{equation}
\label{infD}
\inf_{v\in\widetilde{N}_\omega\setminus H^1(\Rd)}\widetilde{S}(v)\geq K_p>0.
\end{equation}
Finally, the claim follows by combining \eqref{eq-infH1} and \eqref{infD}.
\end{proof}

The latter point is to prove that the right inequality in \eqref{0dwd0w} is strict. To this aim, we mention the following well-known result for the NLS action minimizers at frequency $\omega$ (see, e.g., \cite{C-CL03}).

\begin{proposition}
\label{exuniqS0}
Let $p\in(2,6)$. For every $\omega > 0$, there exists a minimizer of the NLS action at frequency $\omega$. Such minimizer is unique, positive, radially symmetric and decreasing along the radial direction, up to the multiplication by a constant phase and translations.
\end{proposition}

\begin{proposition}
\label{comparinf}
For every $p\in(2,3)$, $\alpha\in\R$ and $\omega>{\omega}_{\alpha}$, there results that $d(\omega)<d^{0}(\omega)$.
\end{proposition}

\begin{proof}
The proof is analogous to that of the first part of Proposition \ref{Eps<Eps0}.
\end{proof}

Concluding the section, we prove the existence part of Theorem {\ref{exchar-actmin}}. 

\begin{proof}[Proof of Theorem \ref{exchar-actmin}-(i)]
The case $\omega\leq {\omega}_{\alpha}$ is a straightforward consequence of the remarks at the beginning of the section and of Corollary \ref{noexistlowfreq}. On the other hand, in order to treat the case $\omega>{\omega}_{\alpha}$ it is sufficient to use Propositions \ref{dpos} and \ref{comparinf} and follow the steps of the proof of \cite[Theorem 1.11]{ABCT-2d}. We mention here a brief sketch for the sake of completeness.

\emph{Step 1: weak convergence of the minimizing sequences.} Let $(u_{n})_{n}$ be a minimizing sequence of the $\delta$-NLS action  at frequency $\omega>{\omega}_{\alpha}$. By Remark \ref{rem-peq}, it is not restrictive to assume that $(u_{n})_n\subset \widetilde{N}_\omega$ and $\widetilde{S}(u_{n})\to d(\omega)$. Now, since $\|u_{n}\|_{p}^{p}\to \f{2p}{p-2}d(\omega)$, $u_{n}$ is bounded in $L^p(\Rd)$. Hence, as $I_{\omega}(u_{n})\leq 0$, setting for instance $\la=\f{\omega+{\omega}_\alpha}{2}$, 
and using the decomposition $u_{n}=\phi_{n,\la}+q_{n}\G_{\la}$, one gets  that $\phi_{n,\la}$ and $u_{n}$ are bounded in $L^2(\Rd)$ and that $q_{n}$ is bounded in $\C$. Thus, there exists $\phi_{\la}\in H^{1}(\Rd)$, $q\in \C$ and $u\in D$ such that $u=\phi_{\la}+q\G_{\la}$ and
\begin{equation*}
\nabla\phi_{n,\la}\deb\nabla\phi_{\la},\quad\phi_{n,\la}\deb\phi_{\la}\quad u_{n}\deb u \quad\text{in}\quad L^{2}(\Rd)\qquad \text{and} \qquad q_{n}\to q\quad\text{in}\quad\C.
\end{equation*}

\emph{Step 2: $u\in D\setminus H^{1}(\Rd)$.} Suppose, by contradiction, that $u\in H^{1}(\Rd)$, so that $q=0$, and define the sequence $w_{n}:=\sigma_{n}\phi_{n,\la} \in H^{1}(\Rd)$ in such a way that $I^{0}_{\omega}(\sigma_{n}\phi_{n,\la})=0$. It is possible to prove that $(\sigma_n^p)_n$ is bounded from above by a sequence $(a_n)_n$ converging to 1. Thus, as $I^{0}_{\omega}(w_{n})=0$ and $\widetilde{S}(u_{n})\to d(\omega)$,
\begin{equation*}
d^0(\omega)+o(1)=\widetilde{S}(w_{n})=\sigma_{n}^{p}\widetilde{S}\left(\phi_{n,\la}\right)\leq  a_{n}\left(\widetilde{S}(u_{n})+o(1)\right)=\widetilde{S}(u_{n})+o(1)=d(\omega)+o(1),
\end{equation*}
that implies that $d(\omega)\geq d^{0}(\omega)$, which contradicts Proposition \ref{comparinf}. 

\emph{Step 3: $u\in \widetilde{N}_\omega$.} In view of Step 2, it is left to prove that $I_{\omega}(u)\leq 0$. Assume by contradiction that $I_{\omega}(u)>0$. Since $u_n$ is bounded in $L^p(\Rd)$, using Brezis-Lieb lemma,
\begin{equation}
\label{eq-Sbrezis}
\widetilde{S}(u_{n})-\widetilde{S}(u_{n}-u)-\widetilde{S}(u)\to 0.
\end{equation}
Moreover, since $q_{n}\to q$, $\na\phi_{n,\la}\deb\na\phi_{\la}$, $\phi_{n,\la}\deb\phi_{\la}$ and $u_{n}\deb u$ in $L^{2}(\Rd)$  and $Q_\omega$ is quadratic, we have also that
\begin{equation}
\label{eq-brexisI}
I_{\omega}(u_{n})-I_{\omega}(u_{n}-u)-I_{\omega}(u)\to 0.
\end{equation}
Let us prove now that $I_{\omega}(u_{n})\to 0$. Assume by contradiction that $I_{\omega}(u_{n})\not\to 0$. Without loss of generality, we can suppose that $I_{\omega}(u_{n})\to -\beta$, with $\beta>0$. Consider, then, the sequence $v_{n}:=\theta_{n}u_{n}$ such that $I_{\omega}(v_{n})=0$. An easy computation shows that
\begin{equation*}
\theta_{n}\to \ell<1.
\end{equation*}
As a consequence,
\begin{equation*}
\widetilde{S}(v_{n})=\widetilde{S}(\theta_{n}u_{n})=\theta_{n}^{p}\widetilde{S}(u_{n})\to \ell^{p}d(\omega)<d(\omega),
\end{equation*}
which is a contradiction. Hence $I_{\omega}(u_{n})\to 0$. Finally, looking back at \eqref{eq-brexisI}, since $I_{\omega}(u)>0$ and $I_{\omega}(u_{n})\to 0$, 
\begin{equation*}
I_{\omega}(u_{n}-u)=I_{\omega}(u_{n})-I_{\omega}(u)+o(1)=-I_{\omega}(u)+o(1),
\end{equation*}
entailing that $I_{\omega}(u_{n}-u)\to-I_{\omega}(u)<0$. Choose then $\bar{n}$ such that $I_{\omega}(u_{n}-u)<0$ for every $n\geq\bar{n}$. Since $d(\omega)\leq \widetilde{S}(u_{n}-u)$ and $\widetilde{S}(u)>0$, \eqref{eq-Sbrezis} yields
\begin{equation*}
d(\omega)\leq \lim_{n}\widetilde{S}(u_{n}-u)= d(\omega)-\widetilde{S}(u)<d(\omega),
\end{equation*}
which is again a contradiction. Thus, $I_\omega(u)\leq 0$.

\emph{Step 4: conclusion.} By the boundedness in $L^p(\Rd)$, $u_n\deb u $ in $L^p(\Rd)$, and so, by weak lower semicontinuity
\begin{equation*}
\widetilde{S}(u)\leq\liminf_{n\to+\infty} \widetilde{S}(u_{n})=d(\omega),
\end{equation*}
which concludes the proof.
 \end{proof}
 


\section{Further properties of ground states and action minimizers}
\label{sec-char}

This final section discusses point (ii) of Theorems \ref{exchar-gs} and \ref{exchar-actmin}. We start by proving (ii)(a) of Theorem \ref{exchar-actmin}.

\begin{proposition}
\label{GSnornos}
Let $p\in(2,3)$, $\alpha \in \R$, $\omega>{\omega}_{\alpha}$ and $u$ be a minimizer of the $\delta$-NLS action at frequency $\omega$. Then, $q\neq0$ and $\phi_\la:=u-q\G_\la\neq 0$, for every $\la>0$.
\end{proposition}

\begin{proof}
Consider the decomposition $u=\phi_\la+q\G_\la$ for a fixed $\la>0$, and suppose by contradiction that $\phi_{\la}=0$. As $u\neq 0$, it must be $q\neq 0$. In addition, $u$ must also fulfil \eqref{eq-regbound}, so that $\alpha+\f{\sqrt{\la}}{4\pi}=0$. Now, whenever $\alpha\geq 0$, this is  a contradiction. On the other hand, when $\alpha<0$, the previous equality entails that $\la={\omega}_\alpha$. However, since $u$ must also satisfy \eqref{EL3}, easy computations yield
\[
 \omega-{\omega}_\alpha+|q|^{p-2}|\G_{{\omega}_\alpha}(x)|^{p-2}=0,\qquad\forall x\in\Rd\setminus\{0\},
\]
which is a contradiction.

It is then left to show that  $q\neq0$. Suppose by contradiction that $q=0$. This would imply that $d(\omega)=d^{0}(\omega)$, which denies Proposition \ref{comparinf}.
\end{proof}

In order to prove point (ii)(b) of Theorem \ref{exchar-actmin}, we recall preliminarily that, up to the multiplication by a phase factor, a $\delta$-NLS action minimizer $u=\phi_{\la}+q\G_{\la}$ can be assumed to display a charge $q>0$ (for details see \cite[Section 5]{ABCT-2d}). In addition, we have to turn to the following equivalent minimum problem (for details see \cite[Proposition 5.3]{ABCT-2d}).

\begin{proposition}
\label{equivprob3}
Let $p\in(2,3)$, $\alpha\in\R$ and $\omega>{\omega}_{\alpha}$. Then,
\begin{equation}
\inf_{v\in D_{C(\omega)}^p}Q_{\omega}(v)=C(\omega),
\end{equation}
with $C(\omega):=\f{2p}{p-2}d(\omega)$ and $D_{C(\omega)}^p:=\{v\in D:\|v\|_p^p=C(\omega)\}$, and there exists $u\in D_{C(\omega)}^p$ such that $Q_{\omega}(u)=C(\omega)$. In particular,
\begin{equation}
\label{eq-fineq}
 \left\{
 \begin{array}{l}
  \displaystyle Q_{\omega}(w)=C(\omega)\\[.2cm]
  \displaystyle w\in D_{C(\omega)}^p
 \end{array}
 \right.
 \qquad\Longleftrightarrow\qquad
 \left\{
 \begin{array}{l}
  \displaystyle S_{\omega}(w)=d(\omega)\\[.2cm]
  \displaystyle w\in N_{\omega}
 \end{array}
 \right..
\end{equation}
\end{proposition}

\begin{remark}
\label{GSequiv}
The main consequence of this result is that we can study the properties of $\delta$-NLS action minimizers at frequency $\omega$ by studying the properties of the minimizers of $Q_\omega$ on $D_{C(\omega)}^p$, with $C(\omega)=\f{2p}{p-2}d(\omega)$.
\end{remark}

Now we can prove the first part of (ii)(b), which states positivity up to multiplication by a constant phase.

\begin{proposition}
\label{realposgs}
Let $p\in(2,3)$, $\alpha\in\R$ and $\omega>{\omega}_{\alpha}$.
Then, every minimizer of the $\delta$-NLS action at frequency $\omega$ is positive, up to multiplication by a constant phase.
\end{proposition}

\begin{proof}
Let $u$ be a minimizer of the $\delta$-NLS action  at frequency $\omega$. As mentioned before, up to multiplication by a constant phase, we can assume $q>0$. On the other hand, by Proposition \ref{equivprob3}, it is also a minimizer of $Q_{\omega}$ on $D_{C(\omega)}^p$, with $C(\omega)=\f{2p}{p-2}d(\omega)$. Now, set $\la=\omega$ and $\Omega:=\{x\in\Rd:\phi_\omega(x)\neq0\}$. By Proposition \ref{GSnornos}, $|\Omega|>0$. As a consequence
\begin{equation*}
u(x)=\phi_{\omega}(x)+q\G_{\omega}(x)=e^{i\eta(x)}|\phi_{\omega}(x)|+q\G_{\omega}(x),\qquad\forall x\in\Omega\setminus\{0\},
\end{equation*}
for some $\eta:\Omega\to[0,2\pi)$. Hence, showing that $\eta(x)=0$ for a.e. $x\in\Omega\setminus\{0\}$ implies that $\phi_\omega(x)=|\phi_\omega(x)|\geq0$ for every $x\in\Rd$, whence $u(x)>0$ for every $x\in\Rd\setminus\{0\}$.

Suppose by contradiction that $\eta\neq 0$ on $\Omega_1\subset(\Omega\setminus\{0\})$, with $|\Omega_1|>0$ and define $\widetilde{u}:=|\phi_{\omega}|+q\G_{\omega}$ (which coincides with $u$ in $\Rd\setminus\Omega_1$). Easy computations yield
\begin{multline*}
|u(x)|^{2}=|\phi_{\omega}(x)|^{2}+q^{2}\G_{\omega}^2(x)+2\cos(\eta(x))|\phi_{\omega}(x)|\G_{\omega}(x)\\[.2cm]
<|\phi_{\omega}(x)|^{2}+q^{2}\G_{\omega}^2(x)+2|\phi_{\omega}(x)|\G_{\omega}(x)=|\widetilde{u}(x)|^{2},\qquad\forall x\in\Omega_1,
\end{multline*}
so that, since $|\Omega_1|>0$,
\begin{equation}
\label{pnormstrict}
\|u\|_{p}^{p}=\int_{\Rd}\left(|u|^{2}\right)^{\f{p}{2}}\dx<\int_{\Rd}\left(|\widetilde{u}|^{2}\right)^{\f{p}{2}}\dx=\|\widetilde{u}\|_{p}^{p}.
\end{equation}
On the other hand, one can simply verify that $Q_{\omega}(\widetilde{u})\leq Q_{\omega}(u)$. Thus, from \eqref{pnormstrict} and the positivity of $Q_\omega$, there exists $\beta\in(0,1)$ such that $\|\beta\widetilde{u}\|_{p}^{p}=\|u\|_{p}^{p}=C(\omega)$ and
\begin{equation*}
Q_{\omega}(\beta\widetilde{u})=\beta^{2}Q_{\omega}(\widetilde{u})<Q_{\omega}(u),
\end{equation*}
which contradicts the fact that $u$ minimizes $Q_\omega$ on $D_{C(\omega)}^p$. As a consequence $\eta=0$ a.e. on $\Omega\setminus\{0\}$, which concludes the proof.
\end{proof}

Note that the arguments before also imply that the regular part of a minimizer of the $\delta$-NLS action  at frequency $\omega$ is nonnegative when $\la=\omega$. In addition, we can prove that, whenever $\la>\omega$, it is in fact positive.

\begin{corollary}
\label{regpos}
Let $p\in(2,3)$, $\alpha\in\R$, $\omega>{\omega}_{\alpha}$ and $u$ be a minimizer of the $\delta$-NLS action at frequency $\omega$. Then the regular part $\phi_{\la}:=u-q\G_\la$ is positive for every $\la>\omega$, up to  multiplication by a constant phase.
\end{corollary}

\begin{proof}
Let $u$ be a positive minimizer of the $\delta$-NLS action at frequency $\omega$ and consider the decomposition $u=\phi_\la+q\G_\la$ for a fixed $\la>\omega$. By \eqref{Gla<Gnu}, since $q>0$, one obtains
\begin{equation*}
\phi_\la(x)=\phi_\omega(x)+q(\G_{\omega}(x)-\G_\la(x))>0,\qquad\forall x\in\Rd\setminus\{0\}.
\end{equation*}
On the other hand, as
\begin{equation*}
\lim_{x\to 0}(\G_\omega-\G_{\la})(x)=\f{\sqrt{\la}-\sqrt{\omega}}{4\pi},
\end{equation*}
the claim is proved.
\end{proof}

The last part of the section is devoted to the radially symmetric monotonicity of minimizers of the $\delta$-NLS action. To this aim, we recall the definition and some important properties of the radially decreasing rearrangement of a function.

Given a measurable $A\subset \R^{3}$ with finite Lebesgue measure, we denote by $A^*$ the open ball centered at zero with Lebesgue measure equal to $|A|$, that is
\begin{equation*}
A^*:=\left\{x\in \R^{3}:\f{4\pi}{3}|x|^3<|A|\right\}.
\end{equation*}
In addition, given $f:\Rd\to\R$ a nonnegative measurable function such that $|\{f>t\}|:=|\{\x\in\Rd:f(\x)>t\}|<+\infty$, for every $t>0$, its radially symmetric rearrangement $f^*:\Rd\to\R$ is defined as
\begin{equation}
f^*(\x)=\int_0^{\infty} \mathds{1}_{\{f>t\}^*}(\x)\dt,
\end{equation}
with $\mathds{1}_{\{f>t\}^*}$ the characteristic function of $\{f>t\}^*$. Concerning such radially symmetric rearrangement we need in the sequel the three properties below. First,
\begin{equation}
\label{equimeas}
\|f^*\|_p=\|f\|_p,\qquad\forall f\in L^p(\Rd),\: f\geq0,\quad \forall p\geq1.
\end{equation}
Second, given two nonnegative functions $f$, $g\in L^p(\Rd)$, with $p>1$, there results 
\begin{equation}
\label{ineqf+g}
\int_{\Rd} |f+g|^p \dx\le \int_{\Rd} |f^*+g^*|^p\dx
\end{equation} 
and, in particular, if $f$ is radially symmetric and strictly decreasing along the radial direction, then the equality in \eqref{ineqf+g} implies that $g=g^*$ a.e. on $\Rd$ (see \cite[Proposition 2.3]{ABCT-2d} for the proof). Finally, if $f\in H^1(\Rd)$, then $f^*\in H^1(\Rd)$ and
\begin{equation}
\label{PS}
\|\na f^*\|_2\leq \|\na f\|_2
\end{equation}
(which is usually called \emph{P\'olya-Szeg\H{o} inequality}).

Using these three properties, we can establish the following proposition.
 
\begin{proposition}
\label{minsymm}
Let $p\in(2,3)$, $\alpha\in\R$ and $\omega>{\omega}_{\alpha}$. Then,
every minimizer of the $\delta$-NLS action  at frequency $\omega$ is radially symmetric and decreasing along the radial directions, up to multiplication by a constant phase.
\end{proposition}

\begin{proof}
Assume, without loss of generality, that $u$ is a positive minimizer of the $\delta$-NLS action at frequency $\omega$ and fix the decomposition $u=\phi_\omega+q\G_\omega$, with $\la=\omega$. We have to show that $\phi_\omega=\phi_\omega^*$. Suppose by contradiction that $\phi_\omega\neq\phi_\omega^*$ and define the function $\widetilde{u}=\phi_\omega^*+q\G_\omega$. By \eqref{PS} and \eqref{equimeas}, we have that $\|\na\phi_\omega^*\|_2\leq \|\na\phi_\omega\|_2$ and $\|\phi_\omega^*\|_2=\|\phi_\omega\|_2$, so that
\begin{equation*}
Q_\omega(\widetilde{u})\leq Q_\omega(u).
\end{equation*}
Now, applying \eqref{ineqf+g} with $f=q\G_\omega$ and $g=\phi_\omega$ in the strict case, there results that $\|\widetilde{u}\|_p^p>\|u\|_p^p$. Therefore, as $Q_\omega$ is positive, there exists $\beta<1$ such that $\|\beta \widetilde{u}\|_p^p=\|u\|_p^p$ and
\begin{equation*}
Q_{\omega}(\beta\widetilde{u})=\beta^{2}Q_{\omega}(\widetilde{u})<Q_{\omega}(\widetilde{u})\leq Q_{\omega}(u),
\end{equation*}
but, via Proposition \ref{equivprob3} (arguing as in the proof of Proposition \ref{realposgs}), this contradicts that $u$ is a $\delta$-NLS action minimizer, thus concluding the proof.
\end{proof}

Finally, we put together all the information we have obtained so far to prove point (ii) of Theorems \ref{exchar-gs} and \ref{exchar-actmin}. 

\begin{proof}[Proof of Theorems \ref{exchar-gs} and \ref{exchar-actmin}-(ii)]
Let $u$ be a minimizer of the $\delta$-NLS action  at frequency $\omega>{\omega}_{\alpha}$. Then, by Proposition \ref{GSnornos}, Proposition \ref{realposgs}, Corollary \ref{regpos} and Proposition \ref{minsymm}, $u$ satisfies all the properties stated in item (ii).

Let, then, $p\in(2,3)$ and $u$ be a $\delta$-NLS ground state of mass $\mu$. Combining Lemma \ref{chargs} and point (i) of Theorem \ref{exchar-actmin} one sees that $u$ is also a minimizer of the $\delta$-NLS action  at some frequency $\omega>{\omega}_{\alpha}$ (in particular, $\omega=\mu^{-1}(\|u\|_p^p-Q(u))$). Then, one concludes by point (ii) of Theorem \ref{exchar-actmin}.
\end{proof}


\subsection*{Acknowledgements}

The work has been partially supported by the MIUR project “Dipartimenti di Eccellenza 2018-2022” (CUP E11G18000350001) and by the INdAM GNAMPA project 2020 “Modelli differenziali alle derivate parziali per fenomeni di interazione”.

\end{document}